\documentclass{article}
\usepackage{amsmath,amsfonts,amssymb}

\DeclareFontEncoding{LS1}{}{}
\DeclareFontSubstitution{LS1}{stix}{m}{n}
\DeclareMathAlphabet\mathscr{LS1}{stixscr}{m}{n}

\usepackage{amsthm}
\usepackage{enumitem}
\usepackage[skip=5pt plus1pt, indent=10pt]{parskip}
\usepackage[bottom=3.5cm, right=3.0cm, left=3.0cm, top=3.5cm]{geometry}
\usepackage{upgreek}
\usepackage{IEEEtrantools}

\usepackage[hidelinks]{hyperref}

\newcommand{\Q}{\mathbb{Q}}
\newcommand{\Z}{\mathbb{Z}}
\newcommand{\N}{\mathbb{N}}
\newcommand{\C}{\mathbb{C}}
\newcommand{\R}{\mathbb{R}}

\newcommand{\bigh}{\mathscr{H}}
\newcommand{\M}{\mathscr{M}}
\newcommand{\h}{\mathscr{h}}

\newcommand\set[1]{\left\{#1\right\}}
\newcommand\br[1]{\left(#1\right)}

\newcommand\abs[1]{\left|#1\right|}
\newcommand\norm[1]{\lVert#1\rVert}

\newtheorem{theorem}{Theorem}

\makeatletter\@addtoreset{case}{theorem}\makeatother
\newtheorem{lemma}{Lemma}
\newtheorem{cor}{Corollary}

\newtheorem{proposition}{Proposition}

\usepackage[
	backend=biber,
	style=alphabetic,
	sorting=nyt
]{biblatex}

\title{On a conjecture of Levesque and Waldschmidt}

\author{
    Tobias Hilgart \small tobias.hilgart@plus.ac.at \\
    Volker Ziegler \small volker.ziegler@plus.ac.at \\
    \small University of Salzburg, Austria
}

\begin{document}

\maketitle

\section{Motivation and Statement}

Ever since the work of Baker \cite{Bak68}, we know how to effectively solve any Thue equation, i.e. a diophantine equation $F(X, Y) = m$ over the integers, where $F$ is an irreducible homogenous polynomial of degree at least $3$. And if we can solve one, we can solve a finite number of them.

Thomas \cite{Tho90} was amongst the first to study infinitely many (non-binary) Thue equations by looking at parametrised equations $F_n(X, Y) = \pm 1$, where $F_n \in \Z[n][X, Y]$ gives a Thue equation for each $n \in \N$. He proved that
\begin{equation}\label{eq: thomaspolynomial}
    F_n(X, Y) = X^3 - (n-1)X^2 Y - (n+2) XY^2 - Y^3 = \pm 1
\end{equation}
has only a few trivial solutions for large parameters $n$ while giving a complete list of all solutions for small parameters $n$. Said parameter can furthermore be taken to be an integer due to the symmetry $F_{-n-1}(X, Y) = F_n(-Y, -X)$.

Levesque and Waldschmidt \cite{LeWa15} took this equation one step further and \textit{twisted} it by an exponential parameter $t$ in the following way: We can write Thue equations as norm-form equations, e.g. we first factorise the polynomial $F_n(X,1)$ of Equation~\eqref{eq: thomaspolynomial} into its complex roots $\uplambda_0, \uplambda_1, \uplambda_2$ (and forego to denote their dependency on the parameter $n$). Then Equation~\eqref{eq: thomaspolynomial} can be written as
\[
    N_{K/\Q}\br{ X - \uplambda_0 Y } = \pm 1,
\]
where $K = \Q\br{\uplambda_0}$ and $N_{K/\Q}$ denotes the norm relative to $K/\Q$. Levesque and Waldschmidt \textit{twisted} this equation by an integer $t$ and looked, among other things, at the equation
\[
    N_{K/ \Q}\br{ X - \uplambda_0^t Y } = \pm 1,
\] for which they managed to prove that there exists an effectively computable number $\upkappa$ such that any solution $(x,y,n,t)$ to the above equation with $\max\br{\abs{x}, \abs{y}} \geq 2$ satisfies $\max\br{\abs{x}, \abs{y}, \abs{n}, \abs{t}} < \upkappa$, i.e. there are but finitely many solutions and we, at least in theory, could find them all.

One conjecture Levesque and Waldschmidt posed in their paper is whether a \textit{twist} by multiple exponential parameters, i.e. the norm-form equation
\[
    N_{K/\Q}\br{ X - \uplambda_0^s \uplambda_1^t Y } = \pm 1
\] can be tackled analogously, working with the existing tools used to study Thue equations---a conjecture we want to give a partial positive answer for in this paper. Our theorem is as follows:

\begin{theorem}\label{thm: main}
    Let $\uplambda_0 = \uplambda_0(n), \uplambda_1 = \uplambda_1(n), \uplambda_2 = \uplambda_2(n)$ be the roots of the integer polynomial
    \[
        f(X) = f(X;\, n) = X^3-(n-1)X^2-(n+2)X-1
    \]
    and
    \[
        F(X,Y) = F(X,Y;\, n,s,t) = \br{ X - \uplambda_0^s\uplambda_1^t Y }\br{ X - \uplambda_1^s\uplambda_2^t Y }\br{ X - \uplambda_2^s\uplambda_0^t Y }.
    \]

    Let $(x,y;n,s,t) \in \Z^5$ be a solution to the Thue-Equation
    \begin{equation}\label{eq: main-thue}
        \abs{ F(X,Y) } = 1,
    \end{equation}
    with $\abs{y} \geq 2, n \geq 3$, $st \neq 0$ and let us assume that for a given $\upvarepsilon > 0$ we have
    \begin{equation}\label{eq: main-cond}
        \min\br{ \abs{2s-t}, \abs{2t-s}, \abs{s+t} } > \upvarepsilon \cdot \max\br{\abs{s}, \abs{t}} > 2.
    \end{equation} 
    Then there exists an effectively computable constant $\upkappa > 0$, depending only on $\upvarepsilon$, such that
    \[
        \max\br{ \abs{x}, \abs{y}, n, \abs{s}, \abs{t} } < \upkappa.
    \]
\end{theorem}

\section{Outline of the Proof}\label{sec: outline}

We strive to improve the readability of our paper by separating the conceptual proof from the necessary detail work. Thus we first give an outline of the proof of Theorem~\ref{thm: main}, which follows the general strategy to solve Thue equations via Baker bounds on linear forms in logarithms. At the intermediate steps that require careful technical consideration, we refer to the corresponding lemmas in the subsequent section, where we also collect other auxiliary results that we use in the process.

Let us start by briefly mentioning the degree of restriction we placed on the theoretical solution $(x,y;\, n,s,t)$ of Equation~\eqref{eq: main-thue}. There are infinitely many solutions where $\abs{y} \leq 2$, e.g. $(1,0;\, n,s,t)$, hence the condition $\abs{y}\geq 2$ is necessary. Due to the symmetry $F_{-n-1}(X, Y) = F_n(-Y,-X)$, we only have to consider non-negative $n$, while we drop the $n$ smaller than $3$ for $\log(n)$ and $\log(\log(n))$ to be positive (and indeed defined at all over the real numbers). But we could alternatively take $n \geq 3$ to be the absolute value of the integer parameter and conduct the same proof, where we just have to be mindful of taking $\uplambda_0$ in absolute value as well. Finally, Condition~\eqref{eq: main-cond} is the degree to which we could solve Levesque and Waldschmidt's conjecture and can be read thus: No quotient of two conjugates of $\uplambda_0^s\uplambda_1^t$ is allowed to be too close to $1$, for the whole argument falls apart otherwise.

We want to improve the readability of the terms and equations and thus introduce several short hands. First, we want to hide constants in our expressions that do not affect their asymptotic without immediately resorting to $O$-notations. So when we write $z \ll_{\, \updelta} z'$, we mean that there exists an effectively computable constant $c(\updelta)$ which does not depend on $z$ or $z'$, but which may depend on $\updelta$ such that $z \leq c(\updelta) \, z'$. If both $z \ll z'$ and $z' \ll z$, we write $z = \uptheta\br{z'}$.

Let furthermore $\uptau := \max\br{\abs{s}, \abs{t}}$ and
\begin{alignat}{3}\label{eq: def-alphabeta}
    &\upalpha_0 := \uplambda_0^s\uplambda_1^t, &&\upalpha_1 := \uplambda_1^s\uplambda_2^t,  &&\upalpha_2 := \uplambda_2^s\uplambda_0^t, \nonumber \\
    &\upbeta_0 := x - \upalpha_0 y,\;\; &&\upbeta_1 := x - \upalpha_1 y,\;\;  &&\upbeta_2 := x - \upalpha_2 y.
\end{alignat}
In this notation, $\upalpha_i$, $\upbeta_i$ resp., is the image of $\upalpha_0$, $\upbeta_0$ resp., under the automorphism that maps $\uplambda_0$ to $\uplambda_i$. The numeration of the roots $\uplambda_0, \uplambda_1, \uplambda_2$ is chosen such that $\uplambda_0$ is close to $n$, $\uplambda_1$ is close to $0$ and $\uplambda_2$ is close to $-1$. The exact inequalities used are given here in Lemma~\ref{lem: root-approx}, from Lemma~10 of \cite{LPV99}

One particular $\upbeta_i$ stands out as being very small, or $\frac{x}{y}$ is a good rational approximation to one particular $\upalpha_i$ respectively. We denote this as $\upbeta_j$, i.e. $\abs{\upbeta_j} = \min\br{\abs{\upbeta_0}, \abs{\upbeta_1}, \abs{\upbeta_2}}$ and the remaining two, in no particular order, as $\upbeta_k$ and $\upbeta_l$.

By the triangle inequality, we have that
\begin{equation}\label{eq: betakl-lb}
    2\abs{\upbeta_k} \geq \abs{\upbeta_k - \upbeta_j} = \abs{ y \br{ \upalpha_j - \upalpha_k } },
\end{equation}
same for $\upbeta_l$. We have to be mindful of the difference between two of the algebraic numbers, e.g. $\upalpha_j, \upalpha_k$, not becoming too small, which we do in Lemma~\ref{lem: alphadiff}. Having done that, we can use the above inequality in combination with $\abs{\upbeta_j \upbeta_k \upbeta_l} = \abs{ F(x,y) } = 1$ to derive for an effectively computable constant $c_1 > 0$ the upper bound
\[
    \abs{\upbeta_j} \leq \frac{1}{\abs{\upbeta_k\upbeta_l}} \ll \frac{1}{ \abs{y}^2 \abs{\upalpha_j - \upalpha_k} \abs{\upalpha_j - \upalpha_l} } \overset{\text{Lemma~\ref{lem: alphadiff}}}{\ll} \frac{1}{\abs{y}^2 \max\br{\abs{\upalpha_j}, \abs{\upalpha_k}} \max\br{ \abs{\upalpha_j}, \abs{\upalpha_l} } }.
\]

Using Lemma~\ref{lem: prodbymax}, the simple property that the product of pairwise maxima of positive real numbers multiplying to $1$ is bounded from below by an appropriate root of the global maximum, we get
\[
    \max\br{\abs{\upalpha_j}, \abs{\upalpha_k}} \max\br{ \abs{\upalpha_j}, \abs{\upalpha_l} } \geq \sqrt{\max\br{\abs{\upalpha_0}, \abs{\upalpha_1}, \abs{\upalpha_2}}}.
\]

We ascertain that at least the largest of the algebraic number $\upalpha_i$ behaves as expected, and indeed, by Lemma~\ref{lem: alphamax}, the above square root is at least $n^{\frac{c_1}{2} \uptau}$. We thus get that our bound for $\abs{\upbeta_j}$ is indeed small, namely 
\begin{equation}\label{eq: betaj-ub}
    \abs{\upbeta_j} \ll \frac{1}{\abs{y}^2 \max\br{\abs{\upalpha_j}, \abs{\upalpha_k}} \max\br{ \abs{\upalpha_j}, \abs{\upalpha_l} } } \ll \frac{1}{\abs{y}^2 n^{\frac{c_1}{2} \, \uptau} }.
\end{equation}

We then use this information to rewrite Siegel's Identity into a unit equation where we now know one addend to be very small. Stating the identity
\begin{equation}\label{eq: Siegel}
    \upbeta_j \br{ \upalpha_k - \upalpha_l } + \upalpha_l \br{\upalpha_j- \upalpha_k} + \upbeta_k \br{\upalpha_l - \upalpha_j} = 0
\end{equation}
and dividing by the last addend with flipped sign gives
\begin{equation}\label{eq: Siegel-alt}
    \frac{\upbeta_j}{\upbeta_k} \frac{\upalpha_k - \upalpha_l}{\upalpha_j - \upalpha_l} + \frac{\upbeta_l}{\upbeta_k} \frac{\upalpha_j - \upalpha_k}{\upalpha_j - \upalpha_l} = 1.
\end{equation}

The first addend on the left-hand side is very small: We use the inequality from Equation~\eqref{eq: betakl-lb} again, in combination with the bound from Inequality~\eqref{eq: betaj-ub} derived above to get
\begin{align*}
    \abs{ \frac{\upbeta_j}{\upbeta_k} \frac{\upalpha_k - \upalpha_l}{\upalpha_j - \upalpha_l} } \ll \frac{ \abs{\upalpha_k-\upalpha_l} }{ \abs{y}^3 n^{ \frac{c_1}{2} \uptau} \abs{\upalpha_j - \upalpha_k} \abs{\upalpha_j - \upalpha_l} } \ll \frac{ \max\br{\abs{\upalpha_k}, \abs{\upalpha_l}} }{ \abs{y}^3 n^{ \frac{c_1}{2} \uptau} \max\br{\abs{\upalpha_j}, \abs{\upalpha_k}} \max\br{\abs{\upalpha_j}, \abs{\upalpha_l}} } .
\end{align*}

We differentiate the cases for whether $\upalpha_j$ has the maximal absolute value or not in Lemma~\ref{lem: alphasalad} and resolve the quotients of the various maxima with an upper bound of $2\abs{y}$. We thus have
\[
    \abs{ \frac{\upbeta_j}{\upbeta_k} \frac{\upalpha_k - \upalpha_l}{\upalpha_j - \upalpha_l} } \ll \frac{1}{\abs{y}^2 n^{\frac{c_1}{2} \uptau} },
\]
and unless both $n$ and $\uptau$ are small, we can assume this to be smaller than $\frac{1}{2}$.

Note that neither of the two addends in Equation~\eqref{eq: Siegel-alt} is zero: None of the $\upbeta_i$ is since each $\upalpha_i$ is irrational, and any two $\upalpha_p, \upalpha_q$ are distinct. The latter fact of which can be seen by looking at Equation~\eqref{eq: logalphaquotient} from the proof of Lemma~\ref{lem: alphadiff}, together with the fact that $\log\uplambda_0,\log\abs{\uplambda_2}$ are $\Q$-linearly independent, since $\uplambda_0, \uplambda_2$ give a system of fundamental units for $\Z[\uplambda_0]$ by Proposition~\ref{prop: Tho}, and $st\neq 0$.

This gives for Equation~\eqref{eq: Siegel-alt} that
\[
    \frac{\upbeta_l}{\upbeta_k} \frac{\upalpha_j - \upalpha_k}{\upalpha_j - \upalpha_l} = 1 - \frac{\upbeta_j}{\upbeta_k} \frac{\upalpha_k - \upalpha_l}{\upalpha_j - \upalpha_l} \neq 1,
\]
and taking the logarithm, using the relation $\abs{\log\br{1-v}} \leq 2\abs{v}$ for $\abs{v}\leq \frac{1}{2}$, yields
\begin{equation}\label{eq: LF}
    0 < \abs{ \underbrace{ \log\abs{ \frac{\upbeta_l}{\upbeta_k} } + \log\abs{\frac{\upalpha_j - \upalpha_k}{\upalpha_j - \upalpha_l}} }_{=: \Uplambda} } \leq 2 \abs{ \frac{\upbeta_j}{\upbeta_k} \frac{\upalpha_k - \upalpha_l}{\upalpha_j - \upalpha_l} } \ll \frac{1}{\abs{y}^2 n^{\frac{c_1}{2} \uptau} }.
\end{equation}

With some extra effort exerted in Lemma~\ref{lem: logy-ub}, this allows us to directly derive an upper bound for $\log\abs{y}$, without troubling, say, Bugeaud and Gy\H{o}ry's result \cite{BuGy96}, which holds much more generally. We get that
\[
    \log\abs{y} \ll \, \uptau \, \br{\log n}^3 \br{ \log\uptau + \log\log n }.
\]

Returning to the linear form in logarithms $\Uplambda$, we managed to derive the bound for $\log\abs{y}$ by shifting the dependence on the exponents $s,t$ from the logarithm of $\abs{\frac{\upbeta_l}{\upbeta_k}}$ into its coefficient. If we manage to do the same for the second logarithm $\log\abs{\frac{\upalpha_j - \upalpha_k}{\upalpha_j - \upalpha_l}}$ as well, then the tools we used before will give us, also using the bound for $\log\abs{y}$, an absolute bound for the parameters $n$ and $\uptau$, thus proving the theorem.

To that end, we extract the larger of $\upalpha_j$ and $\upalpha_k$, which we call $\upalpha_u$, from the denominator, and $\upalpha_v$, defined the same, from the numerator. We can then separate the factor $\log\abs{\frac{\upalpha_u}{\upalpha_v}}$, and the remaining logarithm is at most $4n^{-c_1\uptau}$ by Lemma~\ref{lem: alphadiff}. We shift this into the upper bound and get a new small linear form in logarithms $\Uplambda'$, where
\begin{equation}\label{eq: lambdaprime}
    \abs{\Uplambda'} = \abs{ \log\abs{ \frac{\upbeta_l}{\upbeta_k} } + \log\abs{ \frac{\upalpha_u}{\upalpha_v} } } \ll \frac{1}{n^{\frac{c_1}{2} \uptau}}. 
\end{equation}

Note that we can no longer argue that $\Uplambda'$ is non-zero by construction. We must consider that it is, which we do at length in Lemma~\ref{lem: non-zeroLF}.

We end up either way with a (possibly different) linear form $\Uplambda''$ in logarithms of $\upbeta_k, \upbeta_l, \upalpha_j,\upalpha_k, \upalpha_l$ and $2$, for which inequalities of the form
\[
    0 < \abs{ \Uplambda'' } < \frac{1}{n^{ \frac{c_1}{2} \uptau}}
\]
hold again. We can now again explicitly shift the dependency on $s,t$ into the coefficients and write $\Uplambda'' = A \log\abs{\uplambda_0} + B \log\abs{\uplambda_2} - C\log 2$, where $A,B$ are linear-combinations of $s,t$ and $C$ is $0$ or $1$. We can use the bounds for the coefficients, heights and $\log\abs{y}$ we already established in Lemmas~\ref{lem: coeff-ub}-\ref{lem: logy-ub} to then use bounds for linear forms in logarithms again and derive in Lemma~\ref{lem: tau-bound} that
\[
    \uptau < c_3 \log n \log\log n.
\]

With this information, we can say that in all our decompositions of $\upalpha_i$ and $\upbeta_i$ into powers of $\uplambda_0, \uplambda_2$, the second power is of no consequence unless $n$ is already sufficiently "small", i.e. $n < \upkappa$, since $\log\abs{\uplambda_2} = \uptheta\br{\frac{1}{n}}$ goes to zero, even pitted against the logarithmic bound for the exponents $s,t$.

Carrying out a rigorous case-differentiation for the type $j$ and the order of the $\abs{\upalpha_i}$ in Lemma~\ref{lem: logy-ub} gives a contradiction to $\abs{y} \geq 2$ in each case under the assumption that $n \geq \upkappa$ is sufficiently large for the powers of $\uplambda_2$ to be of no consequence.

This gives $n < \upkappa$, which gives a bound first for $\uptau$ by Lemma~\ref{lem: tau-bound} and then $\log\abs{y}$ by Lemma~\ref{lem: logy-ub} and concludes the proof of Theorem~\ref{thm: main}.

\par\bigskip

\section{Auxilary Results}\label{sec: aux}

We first define the absolute Weil height and Mahler's measure for the sake of completeness---see, e.g., \cite{Bomb82} or \cite{Smyth08}. If $K$ is a number field of degree $d = [K:\Q]$, and for every place $\upnu$, write $d_\upnu = [K_\upnu : \Q_\upnu]$ for the completions $K_\upnu, \Q_\upnu$ with respect to $\upnu$, we normalise the absolute value $\abs{\cdot}_\upnu$ so that
\begin{enumerate}
    \item if $\upnu \vert p$ for a prime number $p$, then $\abs{p}_\upnu = p^{-d_\upnu / d}$,
    \item if $\upnu \vert \infty$ and $\upnu$ is real, then $\abs{x}_\upnu = \abs{x}^{1/d}$,
    \item if $\upnu \vert \infty$ and $\upnu$ is complex, then $\abs{x}_\upnu = \abs{x}^{2/d}$,
\end{enumerate}
where $\abs{x}$ denotes the Euclidian absolute value in $\R$ or $\C$. Given this normalisation, the product formula
\[
    \prod_{ \upnu } \abs{\upalpha}_\upnu = 1
\]
holds for every $\upalpha \in K^*$. The absolute height of $\upalpha \in K$ is then defined as
\[
    \bigh(\upalpha) = \prod_\upnu \max\br{ 1, \abs{\upalpha}_\upnu},
\]
and the absolute logarithmic height as $\h(\upalpha) = \log \bigh(\upalpha)$. The absolute height is then equal to the Mahler measure $\M(m_\upalpha)$ of its minimal polynomial $m_\upalpha$, i.e. if $m_\upalpha(X) = a_d\prod_{i=1}^d \br{X - \upalpha_i} \in \Z[X]$ is the minimal polynomial of $\upalpha \in K$, then
\[
    \h(\upalpha) = \frac{1}{d} \log \M(m_\upalpha) = \frac{1}{d} \br{ \log\abs{a_d} + \sum_{i=1}^d \log\max\br{1, \abs{\upalpha_i}} }.
\]

\begin{proposition}[Baker, Wüstholz; \cite{BaWh93}]\label{prop: BaWh}
    Let $\upgamma_1, \dots, \upgamma_t$ be algebraic numbers not $0$ or $1$ in $K = \Q\br{\upgamma_1,\dots, \upgamma_t}$, which is of degree $D$. Let $b_1, \dots, b_t \in \Z$ and
    \[
        \Uplambda = b_1 \log\upgamma_1 + \cdots + b_t \log\upgamma_t \neq 0.
    \]

    Then 
    \[
        \log\abs{ \Uplambda } \geq - C \cdot h_1\cdots h_t \cdot \log B,
    \]
    where $C = 18(t+1)!t^{t+1}(32D)^{t+2}\log(2tD)$, $B \geq \max\br{3, \abs{b_1}, \cdots, \abs{b_t} }$ and
    \[
        h_i \geq \max\br{ \h(\upgamma_i), \log\abs{\upgamma_i}\, D^{-1}, 0.16\, D^{-1}  }
    \]
    for $i \in \set{1, \dots, t}$.    
\end{proposition}

After stating the notion of height we use and the fundamental tool--lower bounds for linear form in logarithms--to solve Thue equations, we now start with a few observations that concern our Thue Equations~\eqref{eq: main-thue}.

\begin{lemma}\label{lem: root-approx}
    The roots $\uplambda_0, \uplambda_1, \uplambda_2$ of the polynomial $f(X) = X^3-(n-1)X^2-(n+2)X-1$ satisfy the following inequalities:
    \begin{alignat*}{2}
        n + \frac{1}{n} &< \uplambda_0 &&< n + \frac{2}{n} \\
        - \frac{1}{n+1} &< \uplambda_1 = -\frac{1}{\uplambda_0 + 1} &&< -\frac{1}{n+2} \\
        -1-\frac{1}{n} &< \uplambda_2 = -\frac{\uplambda_0 + 1}{\uplambda_0} &&< -1-\frac{1}{n+1}.
    \end{alignat*}
\end{lemma}
\begin{proof}
    See, for instance, Lemma 10 in \cite{LPV99}
\end{proof}

\begin{cor}\label{cor: root-logs}
    The logarithms of the roots of the polynomial $f(X) = X^3-(n-1)X^2-(n+2)X-1$ satisfy the following inequalities:
    \begin{alignat*}{2}
        \log n &< \log\uplambda_0 &&< \log n + \frac{1}{n^2} \\
        -\log n - \frac{2}{n} &< \log\abs{\uplambda_1} &&< -\log n - \frac{1}{2n} \\
        \frac{1}{n} - \frac{2}{n^2} &< \log\abs{\uplambda_2} &&< \frac{1}{n} + \frac{1}{n^2}.
    \end{alignat*}
\end{cor}

\begin{proposition}[Thomas; \cite{Tho79}]\label{prop: Tho}
    Let $\uplambda_0, \uplambda_1, \uplambda_2$ be the roots of the polynomial $f(X) = X^3 - (n-1)X^2 - (n+2)X -1$ and $K$ be the number field generated by them. Then $\uplambda_0, \uplambda_2$ give a set of fundamental units for the order $\Z[\uplambda_0]$.
\end{proposition}

We now show that any two $\upalpha_p, \upalpha_q$, as defined in Equation~\eqref{eq: def-alphabeta} are separated.

\begin{lemma}\label{lem: alphadiff}
    There is an effectively computable constant $c_1 = c_1(\upvarepsilon) > 0$ which, for sufficiently large $n$, can be chosen to be arbitrarily close to $\upvarepsilon$, so that, for two distinct $\upalpha_p, \upalpha_q$ with $\abs{\upalpha_q} > \abs{\upalpha_p}$ we have $\abs{\frac{\upalpha_p}{\upalpha_q}} < n^{-c_1 \uptau} \leq \frac{1}{2}$ and $\abs{\upalpha_p - \upalpha_q} > \br{1 - n^{-c_1 \uptau} } \abs{\upalpha_q} > \frac{1}{2} \abs{\upalpha_q} $.
\end{lemma}
\begin{proof}
   We have $\abs{\upalpha_p - \upalpha_q} = \abs{\upalpha_q} \br{ 1 - \abs{\frac{\upalpha_p}{\upalpha_q}} }$, i.e. we have to prove that the fraction is at most $\frac{1}{2}$. To that end, we look at the logarithm $-\log\abs{ \frac{\upalpha_p}{\upalpha_q} }$ and bound it from below.

    By Lemma~\ref{lem: root-approx}, we can express everything and especially logarithms of quotients of $\upalpha_i$ in terms of $\log\uplambda_0$ and $\log\abs{\uplambda_2}$.
    \begin{alignat}{5}\label{eq: logalphaquotient}
        \log\abs{\frac{\upalpha_p}{\upalpha_q}} =
        \left\{
        \begin{aligned}
            &&\br{2s-t} \, &\log\uplambda_0 &-\br{2t-s} \, &\log\abs{\uplambda_2} \;\; &\text{ if } \set{p,q} = \set{0,1} \\
            &&\br{s-2t} \, &\log\uplambda_0 &-\br{s+t} \, &\log\abs{\uplambda_2} \;\; &\text{ if } \set{p,q} = \set{0,2} \\
            &&\br{s+t} \, &\log\uplambda_0  &-\br{t-2s} \, &\log\abs{\uplambda_2} \;\; &\text{ if } \set{p,q} = \set{1,2}
        \end{aligned}
        \right.
    \end{alignat}

    We use the reverse triangle inequality to bound the positive $-\log\abs{\frac{\upalpha_p}{\upalpha_q}}$. The absolute value of the coefficient of $\log\abs{\uplambda_2}$ is at most $3\uptau$, whereas Condition~\eqref{eq: main-cond} guarantees that the absolute value of the coefficient of $\log\uplambda_0$ is at least $\upvarepsilon \uptau$. If we put everything together and use Corollary~\ref{cor: root-logs}  to estimate the logarithms, we get
    \[
        - \log\abs{\frac{\upalpha_p}{\upalpha_q}} > \upvarepsilon \uptau \log\uplambda_0 - 3\uptau \log\abs{\uplambda_2} > \upvarepsilon \uptau \log n - \frac{3 \uptau}{n} - \frac{3\uptau}{n^2},
    \]
    and unless both $\uptau$ and $n$ are already small, there exists an effectively computable constant $c_1 > 0$, dependent only on $\upvarepsilon$, and for sufficiently large $n$ arbitrarily close to $\upvarepsilon$, such that
    \[
        \upvarepsilon \uptau \log n - \frac{3 \uptau}{n} - \frac{3\uptau}{n^2} > c_1 \uptau \log n \geq \log 2.
    \]
\end{proof}

\begin{lemma}\label{lem: prodbymax}
    Let $a,b,c$ be positive real numbers with $abc = 1$, then
    \[
        \max\br{a,b} \max\br{a,c} \geq \sqrt{ \max\br{a,b,c} }.
    \]
\end{lemma}
\begin{proof}
    Let $a = \max\br{a,b,c}$, then $abc = 1$ implies that $a \geq 1$ and thus $\max\br{a,b}\max\br{a,c} = a^2 \geq \sqrt{a}$ holds. If instead, after possibly renaming $b$ and $c$, the maximum is $b = \max\br{a,b,c}$ then $a \geq \frac{1}{\sqrt{b}}$ or $c \geq \frac{1}{\sqrt{b}}$, with equality in the case that $a=c$ for $abc = 1$ to hold. This in turn means that $\max\br{a,b}\max\br{a,c} \geq \frac{b}{\sqrt{b}} = \sqrt{b}$.
\end{proof}

\begin{lemma}\label{lem: alphamax}
   We have
   \[
        n^{3 \uptau} \geq \max\br{\abs{\upalpha_0}, \abs{\upalpha_1}, \abs{\upalpha_2}} \geq n^{c_1 \uptau}.
   \]
\end{lemma}
\begin{proof}
    Analogously to the proof of Lemma~\ref{lem: alphadiff}, we have for the maximum $\upalpha$ that
    \begin{alignat*}{4}
        \log\abs{\upalpha}= 
            \left\{
        \begin{aligned}
            &&(s-t)\, &\log\uplambda_0 \; &-t \, &\log\abs{\uplambda_2} \;\; \text{ if } \upalpha = \upalpha_0 \\
            &&-s\, &\log\uplambda_0 \; &+\br{t-s} \, &\log\abs{\uplambda_2} \;\; \text{ if } \upalpha = \upalpha_1 \\
            &&t \, &\log\uplambda_0 \; &+s \, &\log\abs{\uplambda_2} \;\; \text{ if } \upalpha = \upalpha_2.
        \end{aligned}
        \right.
    \end{alignat*}
    Since $\abs{\upalpha}$ is the maximum of $\abs{\upalpha_0}, \abs{\upalpha_1}, \abs{\upalpha_2}$, its logarithm is the largest of the three expressions given above. Since it is either $\abs{-s} = \uptau$ or $\abs{t} = \uptau$, we can infer
    \[
        3\uptau \log n > \log\abs{\upalpha} \geq \uptau \log\uplambda_0 - 2\uptau \log\abs{\uplambda_2} > \uptau \log n - 2\uptau \br{ \frac{1}{n} + \frac{1}{n^2} }
    \]
    in each of the three cases for $\upalpha$. We can certainly bound the right-hand side by $c_1 \uptau \log n$, the same expression as used in Lemma~\ref{lem: alphadiff} but could also use a constant that does not depend on $\upvarepsilon$, such as $0.19$: Since $\uptau \log n - 2\uptau \br{\frac{1}{n}+\frac{1}{n^2}} > 0.19\, \uptau \log n$ if and only if $1-\br{ \frac{2}{n\log n} + \frac{2}{n^2\log n} } > 0.19$ which is true for $n\geq 3$.
\end{proof}

\begin{lemma}\label{lem: alphasalad}
    We have
    \[
        \frac{ \max\br{\abs{\upalpha_k}, \abs{\upalpha_l}} }{ \max\br{\abs{\upalpha_j}, \abs{\upalpha_k}} \, \max\br{\abs{\upalpha_j}, \abs{\upalpha_l}} } < 2\abs{y}.
    \]
\end{lemma}
\begin{proof}
    Irrespective of whether $\upalpha_j$ has the largest absolute value of the three or not, we can bound the quotient by its inverse, i.e.
    \[
        \frac{ \max\br{\abs{\upalpha_k}, \abs{\upalpha_l}} }{ \max\br{\abs{\upalpha_j}, \abs{\upalpha_k}} \, \max\br{\abs{\upalpha_j}, \abs{\upalpha_l}} } < \frac{1}{ \abs{\upalpha_j} }.
    \]

    If $\abs{\upalpha_j} = \max\br{\abs{\upalpha_0}, \abs{\upalpha_1}, \abs{\upalpha_2}}$, then the right-hand side is at most $n^{-c_1 \uptau}$, by Lemma~\ref{lem: alphamax}. And unless $n$ and $\uptau$ are already absolutely bounded, we can assume this to be smaller than any constant, say, $4 \leq 2\abs{y}$.

    Otherwise, we use Inequality~\eqref{eq: betaj-ub}, which by the same argument is smaller still than, say $\frac{1}{2}$, and get
    \[
        \frac{1}{2} > \abs{ x - \upalpha_j y } \geq \abs{x} - \abs{\upalpha_j} \abs{y}.
    \]

    Since $x \neq 0$---the only solutions with $x=0$ have $y=\pm 1$---we have $\abs{x} \geq 1$ and thus the above inequality can be read as
    \begin{equation}\label{eq: alphaj-y}
        \abs{\upalpha_j} > \frac{1}{2 \abs{y}},
    \end{equation}
    which proves the assertion.
\end{proof}

\begin{lemma}\label{lem: coeff-ub}
    Let $\upbeta_0 = \pm \uplambda_0^a \uplambda_2^b$ be the decomposition of the unit $\upbeta_0 \in \Z[\uplambda_0]$ into powers of the fundamental units $\uplambda_0, \uplambda_2$ of the mentioned order. Then we have
    \[
        \max\br{\abs{a}, \abs{b}} = \uptheta\br{ \frac{\log\abs{y}}{\log n} + \uptau }.
    \]
\end{lemma}
\begin{proof}
    By conjugation, we have that $\upbeta_0 = \pm \uplambda_0^a\uplambda_2^b, \upbeta_1 = \pm \uplambda_1^{a} \uplambda_0^{b}$, and $\upbeta_2 = \pm \uplambda_2^a\uplambda_1^{b}$. We take the decompositions of $\upbeta_k$ and $\upbeta_l$ and take the logarithm. This gives the system of linear equations
    \begin{equation}\label{eq: betakl-decomp-LGS}
        M
        \begin{pmatrix}
            a \\ b
        \end{pmatrix}
        =
        \begin{pmatrix}
            \log\abs{\upbeta_k} \\ \log\abs{\upbeta_l}
        \end{pmatrix},
    \end{equation}
    where, depending on $j$,
    \[ 
        M \in \set{
            \begin{pmatrix}
                \log\abs{ \uplambda_1 } & \log\abs{ \uplambda_0 } \\
                \log\abs{ \uplambda_2 } & \log\abs{ \uplambda_1 }
            \end{pmatrix},
            \begin{pmatrix}
                \log\abs{ \uplambda_0 } & \log\abs{ \uplambda_1 } \\
                \log\abs{ \uplambda_1 } & \log\abs{ \uplambda_0 }
            \end{pmatrix},
            \begin{pmatrix}
                \log\abs{ \uplambda_0 } & \log\abs{ \uplambda_2 } \\
                \log\abs{ \uplambda_1 } & \log\abs{ \uplambda_0 }
            \end{pmatrix}
        }.
    \]

    Since $\log\uplambda_0 = \uptheta\br{\log n}, \log\uplambda_1 = \uptheta\br{-\log n}$ and $\log\uplambda_2 = \uptheta\br{\frac{1}{n}}$ by Corollary~\ref{cor: root-logs}, we have that $\abs{\det M} = \uptheta\br{\br{\log n}^2}$ in either case. In particular, the determinant is non-zero, the matrix thus invertible.    

    We then multiply Equation~\eqref{eq: betakl-decomp-LGS} with the inverse $M^{-1}$ and take the column-wise-maximum norm. We can calculate it for $M$ by taking the row-wise-maximum norm of $M$ and dividing by $\abs{\det M}$, which gives a $\uptheta\br{\frac{1}{\log n}}$. We thus get, with the consistency $\norm{ M^{-1} \vec{\upbeta} } \leq \norm{M^{-1}} \cdot \norm{ \vec{\upbeta} }$,
    \[
        \max\br{ \abs{a}, \abs{b} } = \uptheta\br{ \frac{ \max\br{\abs{\log\abs{\upbeta_k}}, \abs{\log\abs{\upbeta_l}}} }{ \log n } }.
    \]    

    Furthermore, we have
    \begin{align}\label{eq: logy-by-logbetak}
        \log\abs{\upbeta_k}  &= \log\abs{x - \upalpha_k y + \upalpha_j y - \upalpha_j y}  \nonumber \\
        &= \log\abs{y} + \log\abs{\upalpha_j - \upalpha_k} + \log\abs{ 1 + \abs{\frac{\upbeta_j}{y \abs{\upalpha_j - \upalpha_k}}} },
    \end{align}
    and since $\abs{\upalpha_j - \upalpha_k} \geq \frac{1}{2} \abs{\upalpha_j} \geq \frac{1}{4} \abs{y}^{-1}$ by Equation~\eqref{eq: alphaj-y}, while $\abs{\upbeta_j} \ll \abs{y}^{-2} n^{-\frac{c_1}{2} \uptau}$ by Inequality~\eqref{eq: betaj-ub}, the last logarithm is very small and can certainly be bounded by $1$. Thus, 
    \[
        \max\br{\abs{\log\abs{\upbeta_k}}, \abs{\log\abs{\upbeta_l}}} = \uptheta\br{ \log\abs{y} + \log\max\br{\abs{ \upalpha_j - \upalpha_k }, \abs{\upalpha_j - \upalpha_l}} }.
    \]

    Finally, we have by Lemma~\ref{lem: alphadiff} that $\abs{ \upalpha_j - \upalpha_k }$ is at least $\frac{1}{2} \max\br{\abs{\upalpha_j}, \abs{\upalpha_k}}$, and it is at most $2$ times the maximum. That is to say,
    \[
        \log\max\br{\abs{ \upalpha_j - \upalpha_k }, \abs{\upalpha_j - \upalpha_l}} = \uptheta\br{ \log\max\br{\abs{\upalpha_j}, \abs{\upalpha_k}, \abs{\upalpha_l}} },
    \]
    which by Lemma~\ref{lem: alphamax} is a $\uptheta\br{ \uptau \log n }$. This gives us that
    \[
        \max\br{\abs{\log\abs{\upbeta_k}}, \abs{\log\abs{\upbeta_l}}} = \uptheta\br{\log\abs{y} + \uptau \log n}
    \]
    and thus
    \[
        \max\br{\abs{a}, \abs{b}} = \uptheta\br{ \frac{\log\abs{y}}{\log n} + \uptau },
    \]
    which proves the assertion.
\end{proof}

\section{Proof of Theorem~\ref{thm: main}}\label{sec: proof}

With most of the technical considerations done in the previous section, we can formally finish the proof of Theorem~\ref{thm: main}, as outlined in Section~\ref{sec: outline}. That is to say, we now use bounds for linear forms in logarithms, namely Proposition~\ref{prop: BaWh}, to deduce bounds for the size of the parameters of our Thue equations.

\begin{lemma}\label{lem: logy-ub}
    We have, for some effectively computable constant $c_2 > 0$,
    \[
        \log\abs{y} < c_2\, \uptau\, \br{\log n}^3 \br{ \log\uptau + \log\log n }.
    \]
\end{lemma}
\begin{proof}
    We take the decomposition $\upbeta_0 = \pm \uplambda_0^a \uplambda_2^b$ and conjugate the expression to get decompositions for $\upbeta_1$ and $\upbeta_2$. Only this time, we express everything in powers of $\uplambda_0$ and $\uplambda_2$ using the connections from Lemma~\ref{lem: root-approx}. This gives that $\upbeta_0 = \pm \uplambda_0^a\uplambda_2^b$, $\upbeta_1 = \pm \uplambda_0^{-a+b}\uplambda_2^{-a}$, and $\upbeta_2 = \pm \uplambda_0^{-b}\uplambda_2^{a-b}$.

    We can thus write for the logarithm of the quotient, $\log\abs{\frac{\upbeta_l}{\upbeta_k}} = a' \log\uplambda_0 + b'\log\abs{\uplambda_2}$, where both $a', b'$ are the respective linear combination of $a,b$ and can write the linear form in logarithms $\Uplambda$ from Equation~\eqref{eq: LF} as
    \[
        0 < \abs{ \Tilde{a} \log\uplambda_0 + \Tilde{b} \log\abs{
        \uplambda_2} + \log\abs{\frac{\upalpha_j - \upalpha_k}{\upalpha_j - \upalpha_l}} } \ll \frac{1}{\abs{y}^2 n^{\frac{c_1}{2} \uptau}}.
    \] 
   
   We now use lower bounds for this linear form in logarithms, Proposition~\ref{prop: BaWh}. 
    
    By Lemma~\ref{lem: coeff-ub}, we have that $\max\br{\abs{a'}, \abs{b'}} \ll \frac{\log\abs{y}}{\log n} + \uptau$. Due to the sub-additivity and sub-multiplicativity of the absolute logarithmic height, we have $\h\br{ \frac{\upalpha_j - \upalpha_k}{\upalpha_j - \upalpha_l} } \ll \uptau \, \h\br{\uplambda_0}$. And switching to the Mahler measure of the minimal polynomial $f$ of $\uplambda_0$ gives
    \begin{equation}\label{eq: heightbound}
        \h\br{\uplambda_0} = \M\br{f} = \frac{1}{3} \br{ \log\br{n+2} + \log\abs{\uplambda_0} + \log\abs{\uplambda_2} } = \uptheta\br{ \log n }.
    \end{equation}

    We plug everything into Proposition~\ref{prop: BaWh} and get that for some large effective constant $c_2 > 0$, since the constant in Proposition~\ref{prop: BaWh} is already larger than $10^{14}$, that
    \[
        -c_2\, \uptau \br{\log n}^3 \log\br{ \frac{\log\abs{y}}{\log n} + \uptau } < - \log\abs{y}
    \]
    holds, which implies the assertion.
\end{proof}

\begin{lemma}\label{lem: non-zeroLF}
    There exists a non-zero linear form $\Uplambda''$ in logarithms of $\upbeta_k, \upbeta_l, \upalpha_0, \upalpha_1, \upalpha_2, 2$ with coefficients $0$ or $\pm 1$, such that
    \[
        0 < \abs{\Uplambda''} < \frac{1}{n^{\frac{c_1}{2} \uptau}}
    \]
    holds.
\end{lemma}
\begin{proof}
    If the linear form $\Uplambda'$ from Equation~\eqref{eq: lambdaprime} is already non-zero, the purported upper bound follows from factoring $\frac{\upalpha_u}{\upalpha_v}$ in Equation~\eqref{eq: LF} and shifting the remainder into the upper bound, which gives
    \[
        \abs{\Uplambda'} \ll \frac{1}{\abs{y}^2n^{ \frac{c_1}{2} \uptau}} + \log\abs{\frac{ 1 - \frac{\upalpha_{u'}}{\upalpha_u} }{1 - \frac{\upalpha_{v'}}{\upalpha_v}}}.
    \]

    Both quotients of two $\upalpha_i$ respectively on the right-hand side are smaller than $ n^{-c_1 \uptau}$ by Lemma~\ref{lem: alphadiff}, which gives the purported upper bound for $\abs{\Uplambda'}$, since $\log\abs{1 \pm n^{-c_1 \uptau}} < n^{-c_1 \uptau}$.
    
    Assume instead that $\Uplambda' = 0$, which holds if and only if $\abs{ \upbeta_l \upalpha_u } = \abs{\upbeta_k \upalpha_v}$. We now differentiate the cases for $(u,v)$ and the sign.

    \emph{Case $\abs{\upalpha_j} > \max\br{\abs{\upalpha_k}, \abs{\upalpha_l}}$:}
    In this case, $\upalpha_u=\upalpha_v=\upalpha_j$ cancels out and we have $\abs{\upbeta_l} = \abs{\upbeta_k}$. They cannot be equal, since $\upbeta_l = \upbeta_k$ implies $\upalpha_l = \upalpha_k$, which gives $s=t=0$ due to the multiplicative independence of $\uplambda_0, \uplambda_2$. Thus, $\upbeta_l = - \upbeta_k$.

    We substitute $\upbeta_l$ for $-\upbeta_k$ in Siegel's Identity~\eqref{eq: Siegel} and obtain
    \[
        \upbeta_j\br{\upalpha_k - \upalpha_l} + \upbeta_k \upalpha_l + 2 \upbeta_k \upalpha_j - \upbeta_k \upalpha_l = 0,
    \]
    which we transform into
    \[
        \frac{\upbeta_j \br{\upalpha_k - \upalpha_l}}{2\upbeta_k \upalpha_j} + \frac{\upalpha_l}{2\upalpha_j} + 1 = \frac{\upalpha_k}{2\upalpha_j}.
    \]
    
    Since $\abs{\upalpha_j}$ is the largest, the right-hand side is not $1$. We can bound, by Lemma~\ref{lem: alphadiff}, $\abs{\frac{\upalpha_l}{2\upalpha_j}} \ll n^{ -\frac{c_1}{2} \uptau}$. By a combination of Inequality~\eqref{eq: betakl-lb}-\eqref{eq: betaj-ub}, and Lemma~\ref{lem: alphadiff}-\ref{lem: alphamax}, we can bound
    \[
        \abs{ \frac{\upbeta_j \br{\upalpha_k - \upalpha_l}}{2\upbeta_k \upalpha_j} } \ll \frac{ \max\br{\abs{\upalpha_k}, \abs{\upalpha_l}} }{ \abs{y}^3 n^{ \frac{c_1}{2} \uptau} \abs{\upalpha_j}^2 } \ll \frac{1}{ \abs{y}^3 n^{ \frac{3c_1}{2} \uptau} }.
    \]

    We thus have that 
    \[
        0 < \abs{ \log\abs{\upalpha_k} - \log\abs{\upalpha_j} - \log 2 } \ll \frac{1}{n^{\frac{c_1}{2} \uptau}}.
    \]

    \emph{Case $\abs{\upalpha_k} > \abs{\upalpha_l} > \abs{\upalpha_j}$:} 
    In this case, we have $\upalpha_u = \upalpha_k$ and $\upalpha_v = \upalpha_l$ and thus $\abs{\upbeta_l \upalpha_k} = \abs{\upbeta_k \upalpha_l}$. If the absolute value cancels without a sign change, the term $\upbeta_k\upalpha_l - \upbeta_l\upalpha_k$ vanishes in Siegel's Identity~\eqref{eq: Siegel}. What remains is
    \[
        \upbeta_j \br{\upalpha_k - \upalpha_l} - \upbeta_k\upalpha_j + \upbeta_l\upalpha_j = 0,
    \]
    or
    \[
        -\frac{\upbeta_j \br{\upalpha_k - \upalpha_l}}{\upbeta_k \upalpha_j} + 1 = \frac{\upbeta_l}{\upbeta_k}.
    \]
    The right-hand side is again not $1$, since $\upalpha_l \neq \upalpha_k$. Analogously to the above case and then using Inequality~\eqref{eq: alphaj-y}, we have that
    \[
            \abs{ \frac{\upbeta_j \br{\upalpha_k - \upalpha_j} }{ \upbeta_k \upalpha_j } } \ll \frac{ \abs{\upalpha_k} }{ \abs{y}^3 n^{ \frac{c_1}{2} \uptau} \abs{\upalpha_k - \upalpha_j} \abs{\upalpha_j} } \ll \frac{1}{ \abs{y}^2 n^{\frac{c_1}{2} \uptau} },
    \]
    and thus
    \[
        0 < \abs{ \log\abs{\upbeta_l} - \log\abs{\upbeta_k} } < \frac{1}{n^{\frac{c_1}{2} \uptau}}.
    \]

    If the absolute value cancels with a sign change, then $\upbeta_k \upalpha_l - \upbeta_l \upalpha_k = -2\upbeta_l\upalpha_k$. Plugging this into Siegel's identity yields
    \[
        -\frac{ \upbeta_j \br{\upalpha_k - \upalpha_l} }{ 2\upbeta_l \upalpha_k } - \frac{\upalpha_j}{2 \upalpha_k} + 1 = - \frac{\upbeta_k\upalpha_j}{2\upbeta_l \upalpha_k}.
    \]

    The right-hand side is not $1$: On the one hand, we have that $\abs{ \upbeta_k \upalpha_j } \ll \abs{y}\abs{\upalpha_k-\upalpha_j}\abs{\upalpha_j} \ll \abs{y} \abs{\upalpha_k}\abs{\upalpha_j}$, and on the other hand that $ \abs{y} \abs{\upalpha_l}\abs{\upalpha_k} \ll \abs{y} \abs{\upalpha_l - \upalpha_j} \abs{\upalpha_k} \ll \abs{2 \upbeta_l \upalpha_k} $. Thus, the right-hand side is effectively bounded by $\abs{ \frac{\upalpha_j}{\upalpha_l} }$, which in this case is smaller than $1$, even taking into account the implied constants. We can similarly effectively bound the terms on the left-hand side by $n^{-\frac{c_1}{2}\uptau}$, from which it follows that
    \[
        0 < \abs{ \log\abs{\upbeta_k} - \log\abs{\upbeta_l} + \log\abs{\upalpha_j} - \log\abs{\upalpha_k} - \log 2 } < \frac{1}{n^{\frac{c_1}{2} \uptau}}.
    \]

    The case where $\abs{\upalpha_k} > \abs{\upalpha_j} > \abs{\upalpha_l}$ follows analogously.
\end{proof}

\begin{lemma}\label{lem: tau-bound}
    There exists an effectively computable constant $c_3 > 0$ such that
    \[
        \uptau < c_3 \log n \log\log n.
    \]
\end{lemma}
\begin{proof}
    By Lemma~\ref{lem: non-zeroLF}, we have a linear form in logarithms of $\upbeta_k, \upbeta_l, \upalpha_j,\upalpha_k,\upalpha_l, 2$ with coefficients in $\set{-1,0,1}$, for which
    \[
        0 < \abs{\Uplambda''} < \frac{1}{n^{ \frac{c_1}{2} \uptau}}
    \]
    holds. We proceed analogously to the proof of Lemma~\ref{lem: logy-ub} and write $\upbeta_0 = \pm\uplambda_0^a \uplambda_2^b$ and so on. We do the same for the $\upalpha_i$, i.e. $\upalpha_0 = \uplambda_0^s\uplambda_1^t = \uplambda_0^{s-t}\uplambda_2^{-t}$. We can then write $\Uplambda''$ alternatively as a linear form in logarithms of $\uplambda_0, \uplambda_2, 2$. 
    
    The coefficients are then linear combinations of $a,b, s,t$ and thus effectively bounded by $\frac{\log\abs{y}}{\log n} + \uptau$ by Lemma~\ref{lem: coeff-ub}. Applying the bound for $\log\abs{y}$ from Lemma~\ref{lem: logy-ub} gives, up to effective constants, the bound 
    \[
        \uptau \br{\log n}^2\br{\log\uptau + \log\log n}.
    \]
    
    The logarithmic heights of $\uplambda_0, \uplambda_2$ are effectively bounded by $\log n$, and the logarithmic height of $2$ is $\log 2$. 

    We plug this into Proposition~\ref{prop: BaWh} and get
    \[
        \uptau \log n \ll \br{\log n}^2 \log\br{ \uptau \br{\log n}^2 \br{ \log \uptau + \log\log n } }.
    \]

    If $\uptau \ll \log n$, then the proposed bound holds in particular. If $\log n \ll \uptau$ instead, then $\log\log n \ll \log\uptau$ as well and the right-hand side of the above inequality becomes
    \[
        \uptau \ll \log n \log\br{ \uptau^3 \log\uptau } \ll \log n \log \uptau,
    \]
    from which $\uptau \ll \log n \log\log n$ follows.
    
\end{proof}

If we use the bound $\uptau < c_3 \log n \log\log n$ and plug it into Lemma \ref{lem: logy-ub}, we get that
\begin{equation}\label{eq: logy-by-n}
    \log\abs{y} < c_4 \br{\log n}^4 \br{ \log\log n }^2,
\end{equation}
and plugging both bounds into Lemma \ref{lem: coeff-ub}, we get
\begin{equation}\label{eq: coeff-by-n}
    \max\br{ \abs{a}, \abs{b} } < c_5 \br{ \log n }^3 \br{ \log\log n }^2.
\end{equation}

All we have left to do is to bound $n$ by an absolute constant, which we do in the following

\begin{lemma}\label{lem: contradiction}
    There exists an effectively computable constant $\upkappa$ such that $n < \upkappa$.
\end{lemma}
\begin{proof}
    We assume $n \geq \upkappa$ is sufficiently large, such that the powers of $\uplambda_2$ do not influence the following arguments.

    We first recall the forms of the $\upbeta_i, \upalpha_i$ in terms of $\uplambda_0,\uplambda_2$, i.e. that $\upbeta_0 = \uplambda_0^a\uplambda_2^b, \upbeta_1 = \uplambda_0^{-a+b}\uplambda_2^{-a}$ and $\upbeta_2 = \uplambda_0^{-b}\uplambda_2^{a-b}$, while $\upalpha_0 = \uplambda_0^{s-t}\uplambda_2^{-t}, \upalpha_1 = \uplambda_0^{-s}\uplambda_2^{-s+t}$ and $\upalpha_2 = \uplambda_0^t\uplambda_2^s$. We differentiate the cases for $j$, i.e. which algebraic number $\upalpha_i$ gets approximated the best by $\frac{x}{y}$, resp. which $\upbeta_i$ is the smallest in absolute value, and $u,v$, i.e. for fixed $j$, which algebraic number is the larger in absolute value between $\upalpha_j, \upalpha_k$ or $\upalpha_j, \upalpha_l$ respectively. We start with the case $j = 0$ and choose $(k,l) = (1,2)$.

    \emph{Case $u = v = 0$:}
    In this case, $\abs{\upalpha_0} = \max\br{\abs{\upalpha_0}, \abs{\upalpha_1}, \abs{\upalpha_2}} > 1$ in particular. Taking the logarithm yields $\br{s-t} \log\uplambda_0 - t \log\abs{\uplambda_2} > 0$. But we have $\abs{t} < c_3 \log n \log\log n$ by Lemma~\ref{lem: tau-bound}, and $\log\uplambda_2 = \uptheta\br{\frac{1}{n}}$ by Corollary~\ref{cor: root-logs}. For $n \geq \upkappa$, the second summand thus does not influence the sign of the expression---we derive $s-t > 0$ or $-s+t < 0$.

    It is also the case that $\log\abs{\frac{\upalpha_u}{\upalpha_v}} = 0$, since $\upalpha_u = \upalpha_v = \upalpha_0$. The absolute value of the linear form $\Uplambda'$ of Equation~\eqref{eq: lambdaprime} is thus
    \begin{equation}\label{eq: LF-trick}
        \abs{\Uplambda'} = \abs{ \log\abs{\frac{\upbeta_1}{\upbeta_2}} } = \abs{ \br{a-2b} \log\uplambda_0 + \br{2a-b} \log\uplambda_2 } < \frac{1}{n^{\frac{c_1}{2} \uptau}}.
    \end{equation}

    Either $a - 2b = 0$, or we can deduce from the above equation, with $\abs{a-2b} \geq 1$ and $\log\uplambda_0 > \log n$, that
    \[
        \log n < \abs{a-2b} \log\uplambda_0 < \abs{2a-b} \log\abs{\uplambda_2} + \frac{1}{n^{\frac{c_1}{2} \uptau}}.
    \]

    But $\log\abs{\uplambda_2} < \frac{1}{n}+\frac{1}{n^2} < \frac{2}{n}$ by Corollary~\ref{cor: root-logs}, and thus
    \[
        \log n < \abs{2a-b} \log\abs{\uplambda_2} + \frac{1}{n^{ \frac{c_1}{2} \uptau}} < \frac{ 2\abs{2a-b}}{n} + \frac{1}{n^{\frac{c_1}{2} \uptau}}.
    \]

    We compare the left-hand and right-hand sides of this inequality and since by Inequality~\eqref{eq: coeff-by-n}
    \[
        \abs{2a-b} < 3 \max\br{ \abs{a}, \abs{b} } < 3 c_5 \br{ \log n }^3 \br{\log \log n}^2
    \]
    find, that it cannot hold if $n \geq \upkappa$ is sufficiently large. Thus, $a-2b = 0$.

    With $a-2b = 0$, Equation~\eqref{eq: LF-trick} becomes
    \[
        \abs{2a - b} \log\abs{\uplambda_2} < \frac{1}{ n^{\frac{c_1}{2} \uptau} },
    \]

    and by Corollary~\ref{cor: root-logs}, if $2a-b \neq 0$ this then implies
    \[
        \frac{1}{n} - \frac{2}{n^2} < \frac{1}{ n^{\frac{c_1}{2} \uptau} }.
    \]

    This gives a contradiction and thus $2a-b = 0$, unless $\frac{c_1}{2} \uptau \leq 1$. However, if $n\geq \upkappa$ is sufficiently large, the constant $c_1$ can be chosen to be arbitrarily close to $\upvarepsilon$ by Lemma~\ref{lem: alphadiff}. So $\frac{c_1}{2} \uptau \leq 1$ implies $\upvarepsilon \uptau \leq 2$ and contradicts Condition~\eqref{eq: main-cond}.
 
    We thus have both $a-2b = 0$ and $2a-b = 0$, which implies $a = b = 0$. But this in turn means that $\upbeta_0 =  \uplambda_0^a \uplambda_2^b = 1$ and thus $x - \upalpha_0 y = 1$. Since $x,y$ are integers and $\upalpha_0$ irrational, this implies $x=1,y=0$ and contradicts $\abs{y} \geq 2$. Our assumption that $n \geq \upkappa$ was thus false.

    \emph{Case $u = 0, v = 2$:} Here, the relation between the sizes of the algebraic numbers $\upalpha_i$ is 
    \[
        1 < \abs{\upalpha_2} \text{ and } \abs{\upalpha_2} > \abs{\upalpha_0} > \abs{\upalpha_1}.
    \]
    Taking the logarithm and using that for $n \geq \upkappa$, the summand with $\log\abs{\uplambda_2}$ is small enough not to influence the inequality relations, we get from this that $0 < t$ and $t > s-t > -s$. This implies $s>0$ in particular.

    Furthermore, $\log\abs{\frac{\upalpha_u}{\upalpha_v}} = \log\abs{\frac{\upalpha_0}{\upalpha_2}} = \br{s-2t}\log\uplambda_0 + \br{-s-t}\log\abs{\uplambda_2}$. We follow the same process as in the previous case, i.e. first assume that the coefficient of $\log\abs{\uplambda_0}$ is non-zero and rewrite
    \[
        \abs{\Uplambda'} = \abs{ \log\abs{\frac{\upbeta_1}{\upbeta_2}} + \log\abs{ \frac{\upalpha_0}{\upalpha_2} } } = \abs{ \br{a-2b+s-2t} \log\abs{\uplambda_0} + \br{2a-b-s-t} \log\abs{\uplambda_2} } < \frac{1}{n^{ \frac{c_1}{2} \uptau}}
    \]
    into
    \[
        \log n < \frac{ 2\abs{2a-b-s-t} }{n} + \frac{1}{n^{ \frac{c_1}{2} \uptau}}.
    \]

    But this cannot hold, since by Lemma~\ref{lem: tau-bound} and Equation~\eqref{eq: coeff-by-n},
    \[
        \abs{2a-b-s-t} \ll \max\br{ \abs{a}, \abs{b} } + \uptau \ll \br{ \log n }^3\br{ \log\log n }^2.
    \]

    The assumption $a-2b+s-2t \neq 0$ was thus false, so the term with $\log\uplambda_0$ vanishes in $\Uplambda'$. 
    
    We then assume that the coefficient $2a-b-s-t$ of $\log\abs{\uplambda_2}$ does not vanish and rewrite $\abs{\Uplambda'} < n^{- \frac{c_1}{2} \uptau}$ into
    \[
        \frac{1}{n} - \frac{2}{n^2} < \abs{ 2a-b-s-t }\log\abs{\uplambda_2} < \frac{1}{n^{\frac{c_1}{2} \uptau}}
    \]
    and comparing the left-hand and right-hand sides gives the contradiction, thus $2a-b-s-t = 0$ as well. We rewrite the equations $a-2b+s-2t = 2a-b-s-t = 0$ to express $a,b$ in terms of $s,t$ and get $a = s, b = s-t$.

    We again use this information on $a,b$ to contradict $\abs{y} \geq 2$. Recall that by Equation~\eqref{eq: logy-by-logbetak}, we can approximate $\log\abs{y}$, up to an error of order $O\br{n^{- \frac{c_1}{2} \uptau}}$, very well by either $\log\abs{\upbeta_k} - \log\abs{\upalpha_j - \upalpha_k}$ or $\log\abs{\upbeta_l} - \log\abs{\upalpha_j - \upalpha_l}$. We are in the case $j=0$ and take a closer look at the approximation for $l=2$, i.e. $\log\abs{\upbeta_2} - \log\abs{\upalpha_0 - \upalpha_2}$. By Equation~\eqref{cor: root-logs}, using Lemma~\ref{lem: alphadiff} gives $\abs{\upalpha_0 - \upalpha_2} = \abs{\upalpha_2} \br{ 1 + O\br{n^{-c_1 \uptau}} }$ and thus
    \begin{align*}
        \log\abs{y} &= \log\abs{\upbeta_2} - \log\abs{\upalpha_2} + O\br{ n^{-\frac{c_1}{2} \uptau} } \\ 
        &= \br{-b-t}\log\uplambda_0 + \br{a-b-s} \log\abs{\uplambda_2} + O\br{ n^{-\frac{c_1}{2} \uptau} } \\
        &= -s \log\uplambda_0 + \br{-s+t} \log\abs{\uplambda_2} + O\br{ n^{-\frac{c_1}{2} \uptau} }.
    \end{align*}
    
    But $-s < 0$ in this case, so if $n \geq \upkappa$ is sufficiently large such that neither the second addend nor the omitted error terms influence the sign, this reads as $\log\abs{y} < 0$ and contradicts $\abs{y} \geq 2$.

    \emph{Case $u=1, v=0$:} The relation between the sizes of the algebraic numbers, and thus between the exponents under the condition that $n \geq \upkappa$, is
    \[
        1 < \abs{\upalpha_1} \text{ and } \abs{\upalpha_1} > \abs{\upalpha_0} > \abs{\upalpha_2} \iff  0 < -s \text{ and } -s > s-t > t,
    \]
    which implies $t < 0$. 
    
    We also write $\log\abs{\frac{\upalpha_u}{\upalpha_v}} = \log\abs{\frac{\upalpha_1}{\upalpha_0}} = \br{-2s+t} \log\uplambda_0 + \br{-s + 2t}\log\abs{\uplambda_2}$. Setting up $\Uplambda'$ and its upper bound and assuming the coefficient of $\log\uplambda_0$ yields
    \[
        \log n < \abs{a-2b-2s+t} \log\abs{\uplambda_0} < \frac{2 \abs{2a-b-s+2t}}{n} + \frac{1}{n^{ \frac{c_1}{2} \uptau} }
    \]
    and gives the contradiction for sufficiently large $n \geq \upkappa$, hence $a-2b-2s+t = 0$.

    Doing the same for the coefficient of $\log\abs{\uplambda_2}$ gives
    \[
        \frac{1}{n} - \frac{2}{n^2} < \abs{2a-b-s+2t} \log\abs{\uplambda_2} < \frac{1}{n^{ \frac{c_1}{2} \uptau}}
    \]
    and the contradiction to assuming $2a-b-s+2t \neq 0$. Hence both terms equal $0$, and we again express $a,b$ in terms of $s,t$ and get $a=t, b=-s$.
    
    We again approximate $\log\abs{y}$ by $\log\abs{\upbeta_2} - \log\abs{\upalpha_0 - \upalpha_2}$. Since $v = 0$, we approximate $\log\abs{\upalpha_0 - \upalpha_2}$ by $\log\abs{\upalpha_0}$ this time, which gives
    \begin{align*}
        \log\abs{y} &= \log\abs{\upbeta_2} - \log\abs{\upalpha_0} + O\br{ n^{-\frac{c_1}{2} \uptau} } \\
        &= \br{-b-s+t}\log\uplambda_0 + \br{a-b+t}\log\abs{\uplambda_2} + O\br{ n^{-\frac{c_1}{2} \uptau} } \\
        &= t \log\uplambda_0 + s\log\abs{\uplambda_2} + O\br{ n^{-\frac{c_1}{2} \uptau} },
    \end{align*}
    if we plug in our expressions for $a,b$. But since $t<0$, this gives $\log\abs{y} < 0$ and contradicts $\abs{y} \geq 2$.

    \emph{Case $u = 1, v=2$:} In this case, we get no chain of inequalities between the algebraic numbers $\upalpha_0,\upalpha_1, \upalpha_2$, but as it turns out, we won't need it.
    
    We write $\log\abs{\frac{\upalpha_u}{\upalpha_v}} = \log\abs{\frac{\upalpha_1}{\upalpha_2}} = \br{-s-t} \log\uplambda_0 + \br{-2s+t} \log\abs{\uplambda_2}$ and plug this into our linear form $\Uplambda'$. We again first argue that $a-2b-s-t = 0$ and then do the same for the coefficient of $\log\abs{\uplambda_2}$, i.e. $2a-b-2s+t = 0$. We deduce $a = s-t, b = -t$. 
    
    Choosing again to approximate $\log\abs{y}$ by $\log\abs{\upbeta_2} - \log\abs{\upalpha_0 - \upalpha_2}$, we have $\log\abs{\upalpha_0 - \upalpha_2} \approx \log\abs{\upalpha_2}$ and thus
    \begin{align*}
        \log\abs{y} &= \log\abs{\upbeta_2} - \log\abs{\upalpha_0 - \upalpha_2} + O\br{n^{-\frac{c_1}{2} \uptau}} \\
        &= \log\abs{\upbeta_2} - \log\abs{\upalpha_2} + O\br{ n^{ -\frac{c_1}{2} \uptau } }
        = \underbrace{\br{-b-t}}_{=0} \log\uplambda_0 + \underbrace{\br{a-b-s}}_{=0} \log\uplambda_2 + O\br{ n^{-\frac{c_1}{2} \uptau} } \\
        &= O\br{ n^{-\frac{c_1}{2} \uptau} }.
    \end{align*}
    
    Thus $\log\abs{y}$ is entirely described by the error $O\br{n^{ -\frac{c_1}{2} \uptau }}$ that we made. But if $n \geq \upkappa$ is sufficiently large, this error is bounded by any constant, say, $\log 2$, and we can thus deduce $\log\abs{y} < \log 2$, which contradicts $\abs{y} \geq 2$.

    The other cases, where $j=1$ or $j=2$, follow analogously. The sub-cases for $u,v$ depend on the particular $j$ and choice for $(k,l)$.
\end{proof}

\par\medskip

By this Lemma \ref{lem: contradiction}, we have deduced an effectively computable constant $\upkappa$, which depends implicitly only on $\upvarepsilon$ (since the constant $c_1$ in Lemma~\ref{lem: alphadiff} does), such that $n < \upkappa$. Plugging this into Lemma~\ref{lem: tau-bound} gives fixed upper bound for $\uptau$, and into Inequality~\eqref{eq: logy-by-n} a bound for $\log\abs{y}$, and thus $\abs{y}$. And if everything else is bounded, so must $\abs{x}$. By abuse of notation, we shall call $\upkappa$ the bound that holds for all these parameters, thus concluding the proof of Theorem~\ref{thm: main}.

\printbibliography

\end{document}